\documentclass[11pt,reqno]{amsart}
\usepackage{amsmath,amsfonts,amsthm,amssymb,color,hyperref}

\usepackage[T1]{fontenc}
\usepackage{pdfsync}

\usepackage{graphicx}

\usepackage{subfigure} 

\makeatletter
     \def\section{\@startsection{section}{1}%
     \z@{.7\linespacing\@plus\linespacing}{.5\linespacing}%
     {\bfseries
     \centering
     }}
     \def\@secnumfont{\bfseries}
     \makeatother

\usepackage{graphicx}

\newcommand{\al}{\alpha}

\newtheorem{theorem}{Theorem}[section]
\newtheorem{lemma}[theorem]{Lemma}

\theoremstyle{definition}

\theoremstyle{remark}
\newtheorem{remark}{Remark}
\numberwithin{equation}{section}
\begin{document}
\title[FSDE with  discontinuous diffusion]{Fractional stochastic differential equation with  discontinuous diffusion}
\author[J.Garz\'on]{Johanna Garz\'on}
 \address{Departamento de Matem\'aticas, Universidad Nacional de Colombia, Carrera 45 No 26-85 Bogot\'a, Colombia. }
  \email{mjgarzonm@unal.edu.co}
\author[J.A. Le\'on]{Jorge A. Le\'on}
\address{Departamento de Control Autom\'atico, Cinvestav-IPN, 
Apartado Postal 14-740, 07000 M\'exico D.F., Mexico}
\email{jleon@ctrl.cinvestav.mx}
\author[S.Torres]{Soledad Torres}
\address{Facultad de Ingenier\'ia, CIMFAV  Universidad de Valpara\'iso, Casilla 123-V, 4059 Valparaiso, Chile. }
 \email{soledad.torres@uv.cl}


\begin{abstract}
In this paper we study a stochastic differential equation driven by a fractional Brownian motion with a discontinuous coefficient. We also give an approximation to the solution of the equation. This is a first step to define a fractional version of the  skew Brownian motion.
\end{abstract}

\maketitle



\allowdisplaybreaks

\section{Introduction}
In this article we show  existence and uniqueness for a stochastic differential equation driven by a fractional Brownian motion with Hurst parameter $H >1/2$ and  a  discontinuous diffusion coefficient.  
Recall that a fractional Brownian motion $B^H$, is a centered Gaussian process with covariance structure  given by:
$$
R_H(t,s) = \mathbb{E}(B_t^H B_s^H) = \frac12 \left\{ |t|^{2H}  + |s|^{2H}| - |t-s|^{2H}  \right\} .
$$ 
The particular case  $H= 1/2$ results the Brownian motion. 

The equation we will consider in this article is  given by:
\begin{equation}\label{DISC}
x_t = x_0 + \int_0^t \sigma (x_s) d B^H_s  , \quad t \geq 0;
\end{equation}
where $\sigma$ is a discontinuous function given by
$$\sigma (x) = \frac{1}{\alpha} 1\{x \geq 0\} + \frac{1}{1-\alpha} 1 \{x < 0\}, $$
for some $\alpha \in (0,1)$. Note that,
without loss of generality, we can assume that $\alpha < 1/2$. 
The following picture shows the function $\sigma$. 


\newpage
\begin{figure}[htbp]
\centering
{\includegraphics[width=53mm] {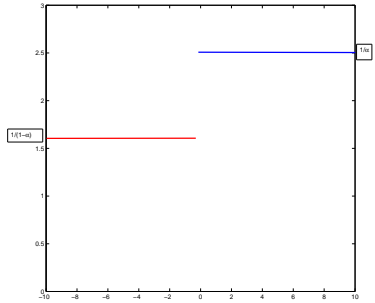}}
\caption{$\sigma(x)$ for $\alpha=0.4$}\label{}
\end{figure}

Nakao \cite{Nakao} in 1972, solved the problem of pathwise  uniqueness of solutions of SDE driven by Brownian motion, ($H=1/2$),  and with a  diffusion coefficient  uniformly positive and of bounded variation on any compact interval. The author proved  the pathwise uniqueness holds for such equation.  After that, many works have been developed in this direction, in order to show existence and/or uniqueness for SDE with general coefficients in the diffusion. We can refer here  \cite{Ouknine} and \cite{Lejay} and the references therein among others. 
 In the case $H > 1/2$, the only cases of stochastic differential equations driven a fractional Brownian motion with discontinuous coefficients which have been studied are those corresponding to discontinuous drift coefficient (for  $H>1/2$).  Regarding that, in  \cite{MN}, the authors studied a H\"older continuous drift except on a finite numbers of points. Other class of discontinuity in SDE driven a fractional Brownian motion is related to add to a Poisson process for the SDE. In \cite{BM}, the authors proved the strong solution of this kind of SDE driven by fBm and Poisson point process extending  the results given in \cite{MN}.

The main interest in working in this type of equations is related with the problem to define a fractional version of the  Skew Brownian motion. In the Brownian motion framework the Skew Brownian motion appeared  as a natural generalization of the Brownian motion. The Skew Brownian motion  is a process that behaves like a Brownian motion except that the sign of each excursion is chosen using an independent Bernoulli random variable of parameter $p \in (0 , 1)$.  For $p = 1/2$ the process corresponds to a Brownian motion. This process is a markov process,  semi-martingale which is a strong solution to some Stochastic Differential Equation (SDE) with local time, (see \cite{Lejay} for a survey).
\begin{equation}\label{SB}
X_t = x + B_t + (2p - 1) L^0_t(X).
\end{equation}
In the case of Brownian motion and  from  It$\hat{o}$-Tanaka formula, equation (\ref{DISC}) and  (\ref{SB}) are equivalent. In the context of fractional Brownian motion, since Tanaka formula only exists for fractional Brownian motion and some functionals on it, there is no relation between both equations.  Is in this sense and until our knowledge this work corresponds to the first step in this direction.


We organized our paper as follows. In Section 2, we give some preliminares related to fractional calculus. Section 3 is devoted to analyze the problem according to the initial condition and the Scheme approximation. The main results is  presented in Section 4. A simulation study is reported in last Section.

\section{Preliminaries}
This section is devoted to introduce an extension of Young's integral defined by
Z\"ahle in \cite{zahle}. Before giving the definition of this integral, we 
establish some notations and definitions that we use in this paper.

Consider $0\le a<b\le T$, $\alpha\in(0,1)$, $p>1$  and $f\in L^{p}([0,T])$. For
 $t\in[a,b]$ we set
\begin{equation}\label{def:der+}
D_{a+}^{\alpha} f({t})
=L^p-\lim_{\varepsilon\downarrow0}\frac{1}{\Gamma(1-\alpha)}
\left(
\frac{f_{t}}{(t-a)^{\alpha}}
+ \alpha \int_{a}^{t-\varepsilon} \frac{f_{t}-f_{r}}{(t-r)^{1+\alpha}} \, dr
\right),
\end{equation}
in case that this limit is well-defined,
where we use the convention $f_r=0$ on $[a,b]^c$. In this case
$D_{a+}^{\alpha} f$ is called the left-fractional derivative of $f$ of order $\alpha$.
Similarly, for   $f\in L^{p}([0,T])$ and $t\in[a,b]$, the right-fractional
derivative of $f$ of order $\alpha$ is introduced as
\begin{equation}\label{def:der-}
D_{b-}^{\alpha} f(t)
=L^p-\lim_{\varepsilon\downarrow0}\frac{1}{\Gamma(1-\alpha)}
\left(
\frac{f_{t}}{(b-t)^{\alpha}}
+ \alpha \int_{t+\varepsilon}^{b} \frac{f_{t}-f_{r}}{(r-t)^{1+\alpha}} \, dr
\right).
\end{equation}
It is not difficult to see that, as a consequence of the proof of
\cite[Theorem 13.2]{Sam}, the fact that $f$, $\frac{f(\cdot)}
{(\cdot-a)^{\alpha}}$ and $\int_a^{\cdot}\frac{f(\cdot)-f_r}{(\cdot-r)^{1+\alpha}}dr$
(resp. $\frac{f(\cdot)}
{(b-\cdot)^{\alpha}}$ and $\int_{\cdot}^b\frac{f(\cdot)-f_r}{(r-\cdot)^{1+\alpha}}dr$)
belong to $L^p([a,b])$ implies that $D_{a+}^{\alpha} f$ (resp.
$D_{b-}^{\alpha} f$) is well-defined and 
\begin{equation}\label{eq:def-Da-plus}
D_{a+}^{\alpha} f({t})
=\frac{1}{\Gamma(1-\alpha)}
\left(
\frac{f_{t}}{(t-a)^{\alpha}}
+ \alpha \int_{a}^{t} \frac{f_{t}-f_{r}}{(t-r)^{1+\alpha}} \, dr
\right),
\end{equation}
(resp.
$$D_{b-}^{\alpha} f({t})
=\frac{1}{\Gamma(1-\alpha)}
\left(
\frac{f_{t}}{(b-t)^{\alpha}}
+ \alpha \int_{t}^{b} \frac{f_{t}-f_{r}}{(r-t)^{1+\alpha}} \, dr
\right)).$$

The space of all the $\alpha$-H\"older continuous functions on $[a,b]$ is denoted by
$C^{\alpha}([a,b])$. Then if $f\in C^{\alpha}([a,b]) $,  the norm of $f$ is defined as follows
$$\left\|f\right\|_{\alpha, [a,b]}:= \left\|f\right\|_{\infty} + \sup_{a\leq s < t \leq b} \frac{|f(t)-f(s)|}{(t-s)^{\lambda}},$$
where $ \left\|f\right\|_{\infty}= \sup_{a\leq t \leq b} |f(t)|$.

 Note that if $f$ belongs to
 $C^{\alpha+\varepsilon}([a,b])$, with
$\varepsilon>0$, then (\ref{eq:def-Da-plus}) is true.

Let $g, f\in L^p([0,T])$ be two functions
 and
 $g_r^{b-}=g_r-g_{b-}$. In this case we say that
$f$ is integrable with respect to $g$ if and only if 
$D_{a+}^{\alpha} f$ and $D_{b-}^{1-\alpha} g^{b-}$ exist, and

$(D_{a+}^{\al} f)D_{b-}^{1-\alpha} g^{b-}\in L^{1}([a,b])$. In this case we 
define the integral $\int_a^b f\, dg$ in the following way
\begin{equation}\label{eq:def-intg-frac}
\int_{a}^{b} f_{r}\, dg_{r}
:=
\int_{a}^{b} (D_{a+}^{\al} f)(r) D_{b-}^{1-\al}g^{b-}(r) \, dr.
\end{equation}

 It is well-known that if $f\in C^{\mu}([a,b])$ and 
 $g\in C^{\beta}([a,b])$, with $\mu+\beta>1$, then it
  can be checked that $\int_{a}^{b} f_{r}\, dg_{r}$ is well-defined, and that it coincides with the Young's integral defined as a limit of Riemann sums.
\section{The Stochastic  Differential Equation }
In this section we  study the solution of the  stochastic differential equation driven by a fBm with a discontinuous diffusion coefficient (\ref{DISC}) as an approximation of a SDE. 
The integral in (\ref{DISC}) is defined pathwise as a Young integral.
\subsection{Initial condition }
We divide our problem according to the initial condition $x_0$ in the following three cases: 
\begin{enumerate}
\item[i)] $ x_0 > 0$,
\item[ii)] $x_0 < 0$,
\item[iii)] $x_0 = 0$.
\end{enumerate}
Note that if  equation (\ref{DISC}) has a continuous solution $x$, then 
it satisfies the following:
{\bf Case i)}  The continuity of $x$ implies that 
 there exits $t_0 > 0$ such that $x_t> 0$ on $(0, t_0)$. Hence, 
\begin{eqnarray*}
x_t = x_0 + \int_0^t \frac{1}{\alpha} d B_s^H = x_0 + \frac {1}{\alpha} B_t^H \ , \quad t \leq t_0.
\end{eqnarray*}
Let $\tilde{t}_0$ be the first instant such that $x_{\tilde{t}_0} = 0$. Then, 
\begin{eqnarray}
\label{eq2}x_t = x_0 + \frac{1}{\alpha} B^H_{\tilde{t}_0} + \int_{\tilde{t}_0}^t \sigma (x_s) d B^H_s= \int_{\tilde{t_0}}^t \sigma (x_s) d B_s^H \ , \quad t \geq \tilde{t}_0.
\end{eqnarray}

{\bf  Case ii)} Similarly,  we have that, there is $t_1 > 0$ such that $x_t < 0$ on $(0 , t_1)$.

Consequently,

\begin{eqnarray*}
x_t = x_0 + \int_0^t \sigma (x_s) d B^H_s = x_0 + \frac {1}{1 - \alpha} B_t^H \ , \quad t \leq t_1.
\end{eqnarray*}

Again, let $\tilde{t}_1$ be the first time such that $x_{\tilde{t}_1} = 0$. Then,
\begin{eqnarray} \label{eq3}
x_t = \int_{\tilde{t}_1}^t \sigma (x_s) d {B}_s^H \quad , \quad t \geq \tilde{t}_1.
\end{eqnarray}

So, we only need to consider the existence of a unique continuous solution to
equation (\ref{eq3}), for any ${\tilde{t}_1}\ge 0$. 


\subsection{ Notation and auxiliary results } In this section,  we introduce the sequence of continuous  functions   $\{\sigma_n : \mathbb{R} \rightarrow \mathbb{R}_+ : n \in \mathbb{N} \}$ converging to $\sigma$ given by:
\begin{eqnarray*}
\sigma_n(x):= \begin{cases} \sigma(x) & \text{if} \ x\notin (-1/n, 0) \\
\frac{1}{\alpha} + n \frac{1 - 2\alpha}{\alpha(1 - \alpha)} x & \text{otherwise}.
\end{cases}
\end{eqnarray*}
Figure 2 shows the approximation function. 
\begin{figure}[h]
{\includegraphics[width=150mm] {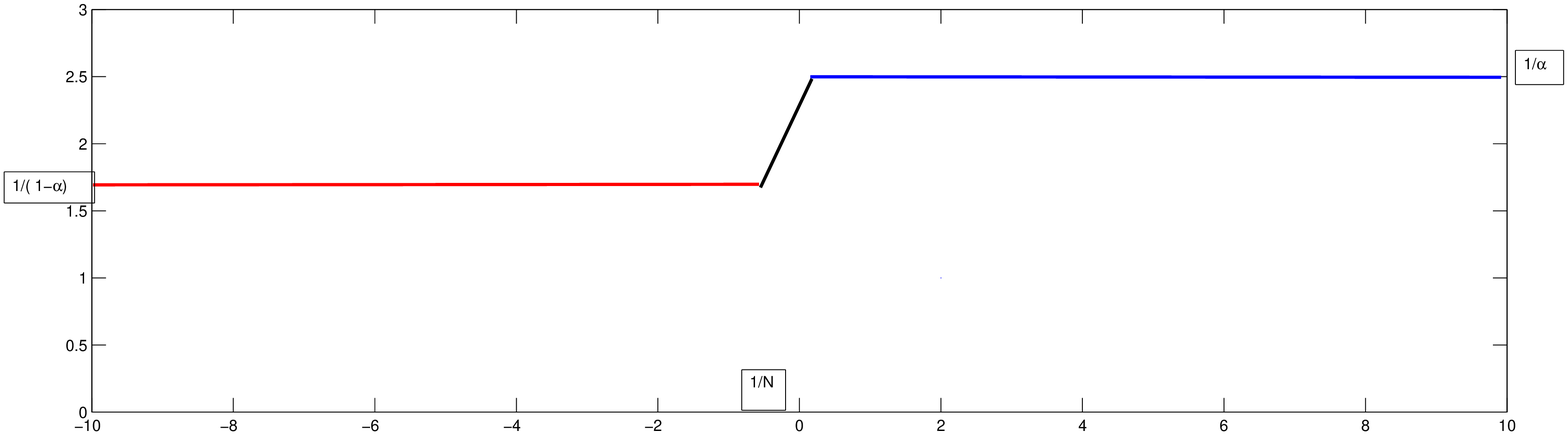}}
\caption{$\sigma^n(x)$ for  $\alpha =0.4$.}
\label{sigman}
\end{figure}

First, note that,  $\sigma_n$ is a Lipschitz continuous function with constant $n \frac{1 - 2\alpha}{\alpha (1 - \alpha)}$,  it means,
\begin{eqnarray}\label{eq4}
| \sigma_n (x) - \sigma_n (y) | \leq n \frac{1 - 2\alpha}{\alpha (1 - \alpha)} |x - y| \quad , \quad x,y \in \mathbb{R}.
\end{eqnarray}

Now we define, for $a \in \mathbb{R}$ fixed, 
\begin{eqnarray}
\label{deflambdan}
\Lambda_n (x) = \int_a^x \frac{ds}{\sigma_n (s)},\quad x\in\mathbb{R}.
\end{eqnarray}
This integral is well defined due to $\sigma_n>0$. Moreover, $\Lambda_n \in C^1(\mathbb{R})$ and $\Lambda'_n (x) = \frac{1}{\sigma_n (x)} > 0$, for all $x\in\mathbb{R}$.
 Therefore $\Lambda_n$ is a strictly increasing function and uniformly bounded in compacts. As a consequence, $\Lambda_n$ has an inverse $\Lambda_n^{-1}$.
 
Set 
\begin{equation}\label{def:L-1}
\Lambda (x) = \int_a^x \frac{ds}{\sigma (s)}  , \quad x \in \mathbb{R}.
\end{equation}

\begin{remark}
Since $\Lambda$ is a strictly increasing function,  then $\Lambda^{-1}$ exists.
\end{remark}

We will need the following auxiliary result later on.

\begin{lemma}\label{propaprox}
\label{lemma1} There exists a positive constant $C_\alpha$
such that, for  $x\in\mathbb{R}$ and $n\in\mathbb{N}$,
$$\left|\Lambda_n(x)-\Lambda(x)\right|\le \frac{C_{\alpha}}{n}\quad
\hbox{\rm and}\quad
\left|\Lambda_n^{-1}(x)-\Lambda^{-1}(x)\right|\le \frac{C_{\alpha}}{n}.$$
\end{lemma}

\begin{proof}
The definitions of $\sigma$ and $\sigma_n$, together with (\ref{deflambdan}) and
(\ref{def:L-1}), imply that
$$\Lambda(x)=\alpha(x-a)1_{\{x\ge 0\}}+((1-\alpha)x-a\alpha)1_{\{x< 0\}}$$
and
\begin{eqnarray*}
\Lambda_n(x)&=&\alpha(x-a)1_{\{x\ge 0\}}\\
&&
+\left(-a\alpha-\frac{\alpha(1-\alpha)}{n(1-2\alpha)}\left(\log(\frac{1}{\alpha})
-\log(\frac{1}{\alpha}+n\frac{1-2\alpha}{\alpha(1-\alpha)}(\frac{-1}{n}\vee x)
\right)\right.\\
&&+\left.(1-\alpha)((\frac{-1}{n}\wedge x)+\frac{1}{n}\right)1_{\{x< 0\}}.
\end{eqnarray*}
Consequently,
\begin{equation}\label{calL-1}
\Lambda^{-1}(x)=\left(\frac{x}{\alpha}+a\right)1_{\{x\ge -a\alpha\}}+
\left(\frac{x+a\alpha}{1-\alpha}\right)1_{\{x\le -a\alpha\}}\end{equation}
and
\begin{eqnarray*}
\Lambda^{-1}_n(x)&=&\left(\frac{x}{\alpha}+a\right)1_{\{x\ge -a\alpha\}}
\\
&&+\frac{1}{\alpha}\left(\exp\left(\frac{n(1-2\alpha)(x+a\alpha)}
{\alpha(1-\alpha)}\right)-1\right)\frac{\alpha(1-\alpha)}{
n(1-2\alpha)}1_{\{\alpha_{n}<x<-a\alpha\}}\\
&&+\left(\frac{x-\alpha_n}{1-\alpha}-\frac{1}{n}\right)
1_{\{x\le \alpha_n\}},
\end{eqnarray*}
with
$$\alpha_n=-a\alpha-\frac{\alpha(1-\alpha)}{n(1-2\alpha)}\left(\log(\frac{1}{\alpha})
-\log(\frac{1}{1-\alpha})\right).$$
Now, it  is easy  to finish  the  proof.

\end{proof}

\begin{remark}
\label{lipschitz} $\Lambda^{-1}$ is a Lipschitz function on $\mathbb{R}$.
Indeed,
\begin{eqnarray}
 \label{eq7}
| \Lambda_n^{-1} (x) - \Lambda_n^{-1} (y) | &\leq & \sup_{z \in K, n\in \mathbb{N}}  | (\Lambda_n^{-1} (z))'| |x - y|\notag\\
&=& \left( \sup_{z \in K, n\in \mathbb{N}} | \sigma_n (\Lambda_n^{-1} (z)) | \right) 
 |x - y|
\notag\\
 &\leq& \left(\frac{1}{\alpha} + \frac{1}{1 - \alpha} \right)  |x-y|. 
\end{eqnarray}

Thus the claim is a consequence of Lemma \ref{lemma1}. 

\end{remark}

We want to see   that $x_t = \Lambda^{-1} (B_t^H - B_a^H + \bar{z})$,
where $\bar{z} = \Lambda (0)$,
 is a solution of the equation 
\begin{eqnarray}
\label{eq8}
x_t = \int_a^t \sigma (x_s) d B_s^H \quad, \quad t \geq a.
\end{eqnarray}
In order to study the uniqueness of the solution to $(\ref{eq8})$, we consider the following auxiliary result.


\begin{lemma}
\label{lemma4} 
Let $\tilde{\alpha} > 1 - H$ and $\gamma < H$ such that $\tilde{\alpha} > 1 - \gamma$. Then,
 \begin{eqnarray*}
| D_{t-}^{1- \widetilde{\alpha}} \left( B^H \right)^{t-}_s| 
 &\leq&
  C_{\tilde{\alpha}} ||B^H||_{\gamma,[0,T]} (t-s)^{\tilde{\alpha} + \gamma - 1}\\
  &\leq&
  C_{\tilde{\alpha}} ||B^H||_{\gamma,[0,T]} T^{\tilde{\alpha} + \gamma - 1},
   \quad 0 \leq s \leq t \leq T,
   \end{eqnarray*}
\end{lemma}
where $(B^H)_s^{t-} = B_s^H - B^H_t$.

\begin{proof}
By  definition of $D_{t-}^{1-\alpha}$ (see equalities  (\ref{def:der+})
and (\ref{eq:def-Da-plus})), we have
\begin{eqnarray*}
\bigg| D_{t-}^{1- \widetilde{\alpha}} \left( B^H \right)_s^{t-} \bigg| &\leq& \frac{1}{\Gamma (\tilde{\alpha})} \bigg( \frac{|B_s^H - B_t^H|}{(t-s)^{1-\tilde{\alpha}}} + (1-\tilde{\alpha}) \int_s^t \frac {|B_s^H - B_r^H|}{(r - s)^{2-\tilde{\alpha}}} dr \bigg)\\
&\leq& C_{\tilde{\alpha}} \bigg( ||B^H||_{\gamma,[0,T]} |t - s|^{\tilde{\alpha} + \gamma - 1} +  ||B^H||_{\gamma,[0,T]} \int_s^t (r - s)^{\tilde{\alpha} + \gamma - 2} dr \bigg)\\
&\leq& C_{\tilde{\alpha}}  ||B^H||_{\gamma,[0,T]} |t - s|^{\tilde{\alpha} + \gamma - 1}. \nonumber \\
\end{eqnarray*} 
Consequently, the proof is complete.
\end{proof}

\subsection{Existence and uniqueness for equation (\ref{DISC}) } 
In this section we state the main results of this article. The first one is related to the existence of a solution of the  SDE (\ref{DISC}), and the second one with the uniqueness
of the solution. 

\begin{theorem}{\bf Existence:}
\label{theorem1}Let $\gamma \in (\frac{1}{2}, H )$. Then there exists a pathwise solution $x \in C^{\gamma} ([0,T])$ to the equation 
\begin{eqnarray}\label{eq18}
x_t = x_0 + \int_0^t \sigma (x_s) d B_s^H , \quad t \in [0,T].  
\end{eqnarray}
\end{theorem}
\begin{proof} By (\ref{eq2}) and (\ref{eq3}), we only need to show that the equation
$$x_t=\int_a^t \sigma(x_s)dB^H_s,\quad t\ge a,$$
has a solution for every $a\ge 0$. Towards  this end, we
observe that, for $\alpha<1/2$, (\ref{calL-1}) implies that $\Lambda^{-1}$ is a convex
function. Then, Remark 3.5 in \cite{AMV} (see also Theorem 2.1 in \cite{MSV})
yields
$$\Lambda^{-1}(B_t^H-B_a^H+\bar{z})-\Lambda^{-1}(\bar{z})=
\int_a^t(\Lambda^{-1})^{'}_+(B_s^H-B_a^H+\bar{z})dB^H_s,\quad t\ge a.$$
Since $(\Lambda^{-1})^{'}_+(x)=\sigma(x-\bar{z})$, setting 
$x_t=\Lambda^{-1}(B_t^H-B_a^H+\bar{z})$, then we have
$$x_t=\int_a^t\sigma(B_s^H-B_a^H)dB^H_s.$$
Finally, we also have that $\sigma(x)=\sigma(\Lambda^{-1}(x
+\bar{z}))$. Thus the proof is complete.

\end{proof}

Now we want to see that ($\ref{eq18}$) has a unique solution $x \in C^\gamma ([0,T])$, for $\gamma \in (\frac{1}{2} , H)$. Towards this end, in the following result we will use the approximation $\sigma_n$, because we will take advantage of  \cite{zahle} (Theorem 4.2.1), 
which requires the Lipschitz property of $\sigma_n$.

Note that we have figure out a solution of ($\ref{eq18}$) such that 
\begin{eqnarray}\label{eq20}
1_{\{a \leq r < s < T\}} \frac{\sigma (x_s) - \sigma (x_r)}{(s-r)^{1+\tilde{\alpha}}} \in L^1([0,T]^2),
\end{eqnarray}
with probability 1.

\begin{lemma}\label{le:lnYi}
Let $\gamma \in (\frac{1}{2} , H )$ and $x$ a $\gamma$-H\"older continuous solution to ($\ref{eq18}$) such that ($\ref{eq20}$) holds. Then, 
\[\Lambda_n (x_t) = \Lambda_n (x_0) + \int_0^t \frac{\sigma (x_s)}{\sigma_n (x_s)} d B_s^H, \quad t \in [0,T].\]
\end{lemma}

\begin{proof}
By \cite{zahle} (Theorem 4.2.1) and (\ref{eq4}), we obtain
 \[\int_0^t \frac{d x_s}{\sigma_n (x_s)} = \lim\limits_{|\pi| \to 0} \sum_{j=0}^{m-1} \frac{x_{s_{i+1}} - x_{s_i}}{\sigma_n (x_{s_i})},\]
where $\pi$ is a partition of [0,t] of the form $\pi = \{0 = s_0 < s_1<\cdots< s_m = t\}.$
Therefore, with the convention $\sigma_n^{\pi} (s) = \sum_{j=1}^{m-1} \frac{1_{[s_i,s_{i+1}]}(s)}{\sigma_n (x_{s_i})}$, we get
\begin{eqnarray}\label{eq21}
 \int_0^t \frac{d x_s}{\sigma_n (x_s)} &=& \lim\limits_{|\pi| \to 0} \sum_{j=0}^{m-1} \int_{s_i}^{s_{i+1}} \frac{\sigma(x_r)}{\sigma_n (x_{s_i})} d B_r^H \notag\\
&=& \lim\limits_{|\pi| \to 0} \int^t_0 \sigma(x_r) \sigma_n^{\pi} (r) dB_r^H.
\end{eqnarray}

On the other hand, the definition of Young integral (\ref{eq:def-intg-frac}) allows to establish
\begin{eqnarray*}\lefteqn{
\left| \int^t_0 \left ( \frac{\sigma(x_s)}{\sigma_n (x_{s})} - \sigma_n^{\pi}(s) \sigma (x_{s}) \right ) dB^H_s \right |} \\
&\le& C_{\tilde{\alpha}} \left | \int^t_0 \sigma (x_{s})( D_{0+}^{\tilde{\alpha}}( \frac{1}{\sigma_n(x_s)} - \sigma_n^{\pi} (s) ) )( D_{t-}^{1 - \tilde{\alpha}} B^H)_{s}^{t-}  ds \right|  \\
&&+ \tilde{\alpha} \left | \int^t_0 \int^s_0 \left( \frac{1}{\sigma_n(x_r)} - \sigma_n^{\pi} (r) \right ) \frac{\sigma(x_s) - \sigma (x_r)}{(s-r)^{1+
\tilde{\alpha}}} dr ( D_{t-}^{1 - \tilde{\alpha}} B^H)_{s}^{t-} ds \right|\\
&\le& C_{\tilde{\alpha},T} ||B^H||_{\gamma, [0,T] } \bigg ( \int^T_0 \left | D_{0+}^{ \tilde{\alpha}} \left (  \frac{1}{\sigma_n(x_s)} - \sigma_n^{\pi}(s) \right) \right| ds \big. \\
&&+  \int^T_0 \int^s_0 \left | \frac{1}{\sigma_n(x_r)} - \sigma_n^{\pi}(r) \right | \frac{  | \sigma(x_s) - \sigma(x_r)  |}{(s-r)^{1+\tilde{\alpha}}} dr ds \bigg ),
\end{eqnarray*}
where last inequality follows from Lemma $\ref{lemma4}$. Then, the result is a consequence of Z\"ahle \cite{zahle} (Theorems 4.1.1 and 4.3.1).
\end{proof}

Note that the solution $x$ to equation ($\ref{eq18}$) is such that, there exists a random variable $G$ such that
\begin{eqnarray*}
 | x_s - x_t  | \le G |s-t|^\gamma.
\end{eqnarray*}
In the following result we set
\begin{eqnarray*}
f_{\gamma }(s) = \mathbb{E}( G^{\frac{\tilde{\alpha} + \varepsilon}{\gamma}} |x_s|^{-\frac{( \tilde{\alpha} + \varepsilon)}{\gamma}}).
\end{eqnarray*}

\begin{lemma}
Let $\gamma \in (\frac{1}{2}, H)$, $1-H<1-\gamma<\tilde{\alpha}<\gamma$, $x$ a $\gamma$-H\"older continuous solution of (\ref{eq18}) such that (\ref{eq20}) 
and $\int_0^T1_{\{x_s=0\}}ds=0$
hold, $f_{\gamma} \in L^1 \left( [0,T]\right)$ for some $\varepsilon>0$ small enough,
 $G^{1-\eta} \in L^q$ and $\mathbb{P} \left( a < x_s \le b \right) \le g(s) (b - a)$, where $g \in L^{\frac{1}{p}} \left( [0,T] \right)$ with $\frac{1}{p} + \frac{1}{q} = 1$ and
 $p\in(1,\frac{1}{1-\eta})$ for some $\eta<1-\frac{{\tilde{\alpha}}}{\gamma}$.
 Then, 
\begin{eqnarray*}
\Lambda (x_t) - \Lambda (x_0) = B_t^H, \ t\in [0,T].
\end{eqnarray*}
\end{lemma}
\begin{remark}
Note that the solution to (\ref{eq18}) given in Theorem \ref{theorem1} satisfies
the assumptions of this result.
\end{remark}
\begin{proof}
From Lemmas $\ref{lemma4}$, (\ref{lemma1}) and \ref{le:lnYi}, we only need to prove  that:
\begin{eqnarray}\label{eq22}
 \int^T_0 s^{-\tilde{\alpha}}  \left| \frac{\sigma(x_s)}{\sigma_n(x_s)} - 1 \right | ds \longrightarrow 0, \quad \text{as} \ n \rightarrow \infty,
\end{eqnarray}
and 
\begin{eqnarray}\label{eq23}
I_{n_k} =\int^T_0 \int^s_0 \left| \frac{\sigma(x_s)}{\sigma_{n_k}(x_s)} - \frac{\sigma(x_r)}{\sigma_{n_k}(x_r)} \right| (s-r)^{-(1+\tilde{\alpha})} dr ds \rightarrow 0, \quad \text{as} \  n_k \longrightarrow \infty,
\end{eqnarray}
for some subsequence$\{n_k:k\in\mathbb{N}\}$, w.p. 1. For ($\ref{eq22}$), we have: 
\begin{eqnarray*}
\int^T_0 s^{-\tilde{\alpha}} \left| \frac{ \sigma(x_s) - \sigma_n(x_s)}{\sigma_n(x_s)} \right| ds \leq C_\alpha \int^T_0 s^{-\tilde{\alpha}} | \sigma(x_s) - \sigma_n(x_s) | ds.
\end{eqnarray*}
Thus the dominated convergence theorem implies that ($\ref{eq22}$) holds.

For (\ref{eq23}), we have that
$$\left|\frac{\sigma(x_s)}{\sigma_n(x_s)} - \frac{\sigma(x_r)}{\sigma_n(x_r)} \right |
\rightarrow 0,\quad \hbox{\rm as }\ n\rightarrow\infty,\ \hbox{\rm w.p. 1},$$
and
\begin{eqnarray}\label{eq24}
\left| \frac{\sigma(x_s)}{\sigma_n(x_s)} - \frac{\sigma(x_r)}{\sigma_n(x_r)} \right | &=& \left| \frac{ \sigma(x_s) \sigma_n(x_r) - \sigma(x_r) \sigma_n(x_s)}{\sigma_n(x_s) \sigma_n(x_r)} \right| \nonumber \\
& \le & C_\alpha  | \sigma(x_s)\sigma_n(x_r) - \sigma(x_r) \sigma_n(x_s)  | \nonumber \\
& \le & C_\alpha  | \sigma(x_s) \big |  | \sigma_n(x_r) - \sigma_n(x_s)  | +  C_\alpha  | \sigma_n(x_s) \big |  | \sigma(x_s) - \sigma(x_r)  | \nonumber \\
& \le & C_\alpha \left( | \sigma_n(x_r) - \sigma_n(x_s) \big | + |  \sigma(x_s) -  \sigma(x_r)  | \right) \nonumber \\
& = & I_1 (s,r) + I_2(s,r).
\end{eqnarray}


Let us start analizing the term  $I_2(s,r)$. In this case we have the following:

\begin{eqnarray}
\label{I1}\lefteqn{
I_{\{0\le r\le s\le T\}} I_2(s,r) (s-r)^{-(1+\tilde{\alpha})} } \nonumber \\
& \le& C_\alpha I_{\{0\le r\le s\le T\}}
  |r-s|^{-1 - \tilde{\alpha}}   \left( 1_{\{x_s < 0 < x_r\} } +
   1_{\{x_r < 0 < x_s\} } \right).
\end{eqnarray}

Since over $ \{ x_s < 0 < x_r  \}$,     $|x_s|  < |x_r - x_s|$ and $|x_r - x_s| 
 \le G |r-s|^{\gamma}$. 
 Also we have  
$ |r - s| ^{-1} \le G^{\frac{1}{\gamma}} |x_s|^{-\frac{1}{\gamma}}$.
This inequality leads to:
\begin{equation}
\label{eq:inq1}
I_{\{0\le r\le s\le T\}} I_2(s,r) (s-r)^{-(1+\tilde{\alpha})} I_{\{x_s<0<x_r\}}
\le C_{\alpha}|r-s|^{-1+\varepsilon}G^{\frac{\tilde{\alpha}+\varepsilon}{\gamma}}
|x_s|^{-\frac{\tilde{\alpha}+\varepsilon}{\gamma}}.
\end{equation}
Similarly we can obtain,
\begin{equation}
\label{eq:inq2}
I_{\{0\le r\le s\le T\}} I_2(s,r) (s-r)^{-(1+\tilde{\alpha})} I_{\{x_r<0<x_s\}}
\le C_{\alpha}|r-s|^{-1+\varepsilon}G^{\frac{\tilde{\alpha}+\varepsilon}{\gamma}}
|x_r|^{-\frac{\tilde{\alpha}+\varepsilon}{\gamma}}.
\end{equation}

Now, we show that
\begin{equation}\label{eq:mG}
I_1(s,r)\le C_{\alpha} n^{1-\eta}G^{1-\eta}(s-r)^{\gamma(1-\eta)} 
\end{equation}
holds.  By 
inequality (\ref{eq4}), we can establish
\begin{eqnarray*}\lefteqn{
 \big| \sigma_n ( x_s) - \sigma_n(x_r) \big|} \\
&=& \big| \sigma_n ( x_s) - \sigma_n(x_r) \big|^{1-\eta}\big| \sigma_n ( x_s) - \sigma_n(x_r) \big|^\eta \\
&=& n^{1-\eta} \left |\frac{1-2\alpha}{(1-\alpha)\alpha} \right|^{1-\eta}  \left|x_s -
x_r
 \right |^{1-\eta}
\big| \sigma_n ( x_s) - \sigma_n(x_r) \big|^\eta \\
&\leq& C_{\alpha, \eta}  n^{1-\eta} \left |x_s-x_r\right|^{1-\eta},
\end{eqnarray*}
which allows us to see that (\ref{eq:mG}) is true.

On the other hand, 
\begin{eqnarray*}\lefteqn{
\mathbb{E}\left(I_{\{x_r<-\frac{1}{n}<x_s<0\}}G^{1-\eta}\right)}\\
&\le&\left(\mathbb{E}G^{q(1-\eta)}\right)^{1/q}\left(\mathbb{P}\left(x_s\in(-\frac{1}{n},0)
\right)\right)^{1/p}
\le Cn^{-\frac{1}{p}}g_s^{\frac{1}{p}}.
\end{eqnarray*}
Hence, (\ref{eq:mG}) implies
$$\mathbb{E}\left(\int_0^T\int_0^sI_1(s,r)(s-r)^{-(\tilde{\alpha}+1)}
I_{\{x_r<-\frac{1}{n}<x_s<0\}}drds\right)\rightarrow 0,$$
as $n\rightarrow\infty$.  Moreover, 
this analysis also implies
$$\mathbb{E}\left(\int_0^T\int_0^sI_1(s,r)(s-r)^{-(\tilde{\alpha}+1)}
\left(I_{\{x_s<-\frac{1}{n}<x_r<0\}}+
I_{\{-\frac{1}{n}<x_r,x_s<0\}}\right)drds\right)\rightarrow 0,$$
as $n\rightarrow\infty$. 
Finally, we can see that inequalities (\ref{eq:inq1}) and (\ref{eq:inq2}) are also true
when we write $I_1(s,r)$ instead of $I_2(s,r)$. Therefore, the result follows
from (\ref{eq24})-(\ref{eq:inq2}) and the dominated convergence theorem.
\end{proof}


\section{Numerical Results}

By (\ref{eq4}) and  the mean value theorem we have that $\sigma_n (\Lambda^{-1}_n(x))$ to be a Lipschitz function, then Z\"ahle (\cite{zahle}, theorem 4.3.1) implies:

\begin{eqnarray} \label{eq5}\lefteqn{
\Lambda_n^{-1}(B_t^H - B_a^H + z_n) - \Lambda_n^{-1} (z_n)}\notag \\
  &=& \Lambda_n^{-1} (B_t^H - B_a^H + z_n) \notag\\
&=& \int_a^t \sigma_n \left(\Lambda_n^{-1} (B_s^H - B_a^H + z_n) \right) d B_s^H,
\quad t\ge a,
\end{eqnarray} 
with $z_n= \Lambda_n (0)$. 

Therefore, $x_t^{(n)} = \Lambda_n^{-1} (B_t^H - B_a^H + z_n) $ is a solution to the equation  
\begin{eqnarray}\label{eq6}
x_t^{(n)} = \int_a^t \sigma_n (x_s^{(n)} )  d B_s^H  , \quad t \geq a.
\end{eqnarray} 

Consequently, by Lemma \ref{propaprox}, we can aproximate the solution of equation (\ref{DISC}) by the solution of equation  (\ref{eq6}).

In this section we show some simulations for the unique solution of equation (\ref{DISC}) for different values  of $H$ and $\alpha$. We use the approximation given in (\ref{eq6}).  We also present the limit case $H=0.5$, and $a=0$.

The following figures correspond to the   particular case the Skew Brownian motion (SBm) as the solution of a SDE with discontinuous coefficient in the  diffusion. Also, the case (c), $\alpha =0.5$  corresponds to the Brownian motion process. 

\begin{figure}[htbp]
\centering
\subfigure[SBm for $\alpha = 0.99$. ]{\includegraphics[width=45mm] {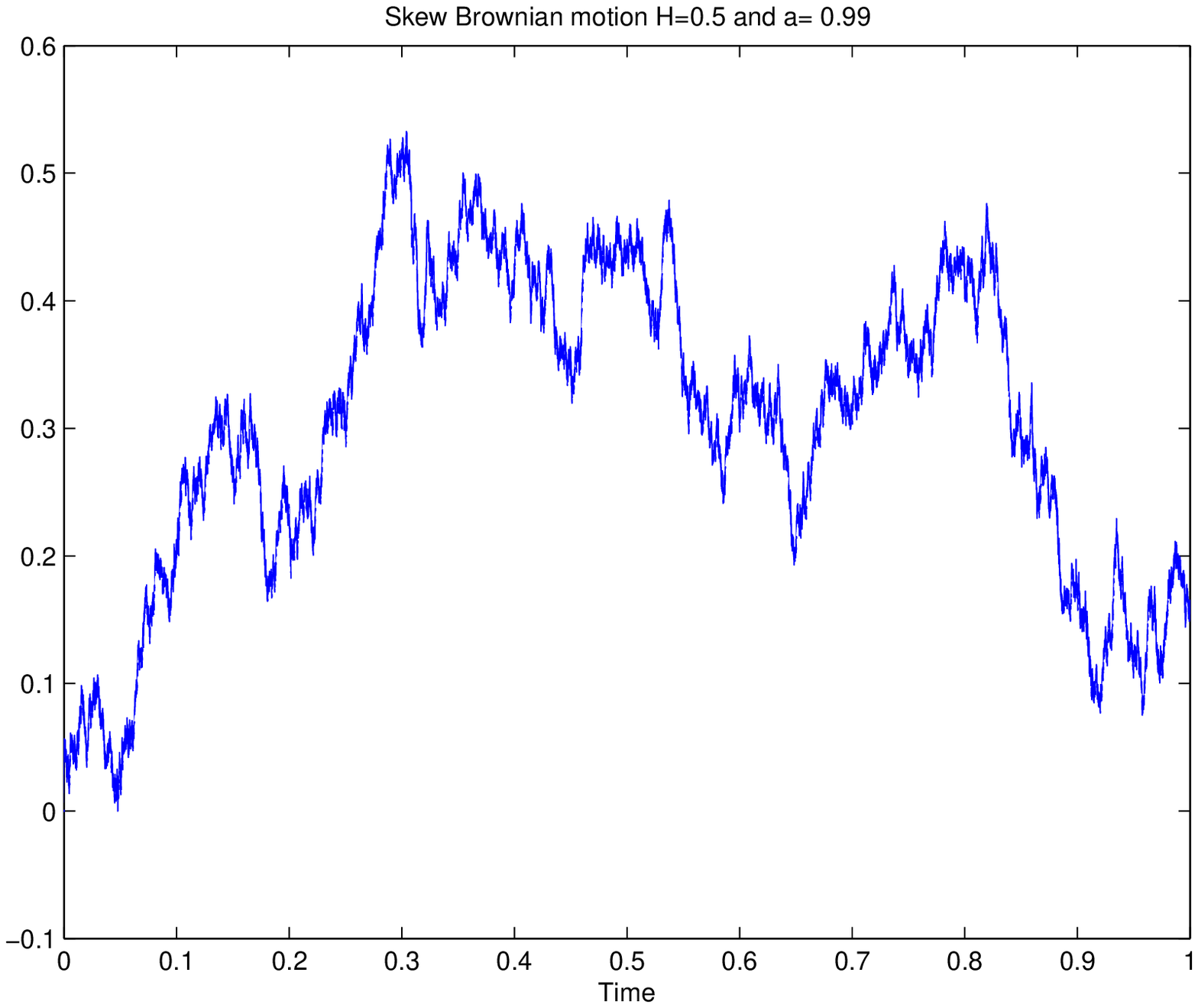}}
\subfigure[SBm for $\alpha = 0.01$.]{\includegraphics[width=45mm]{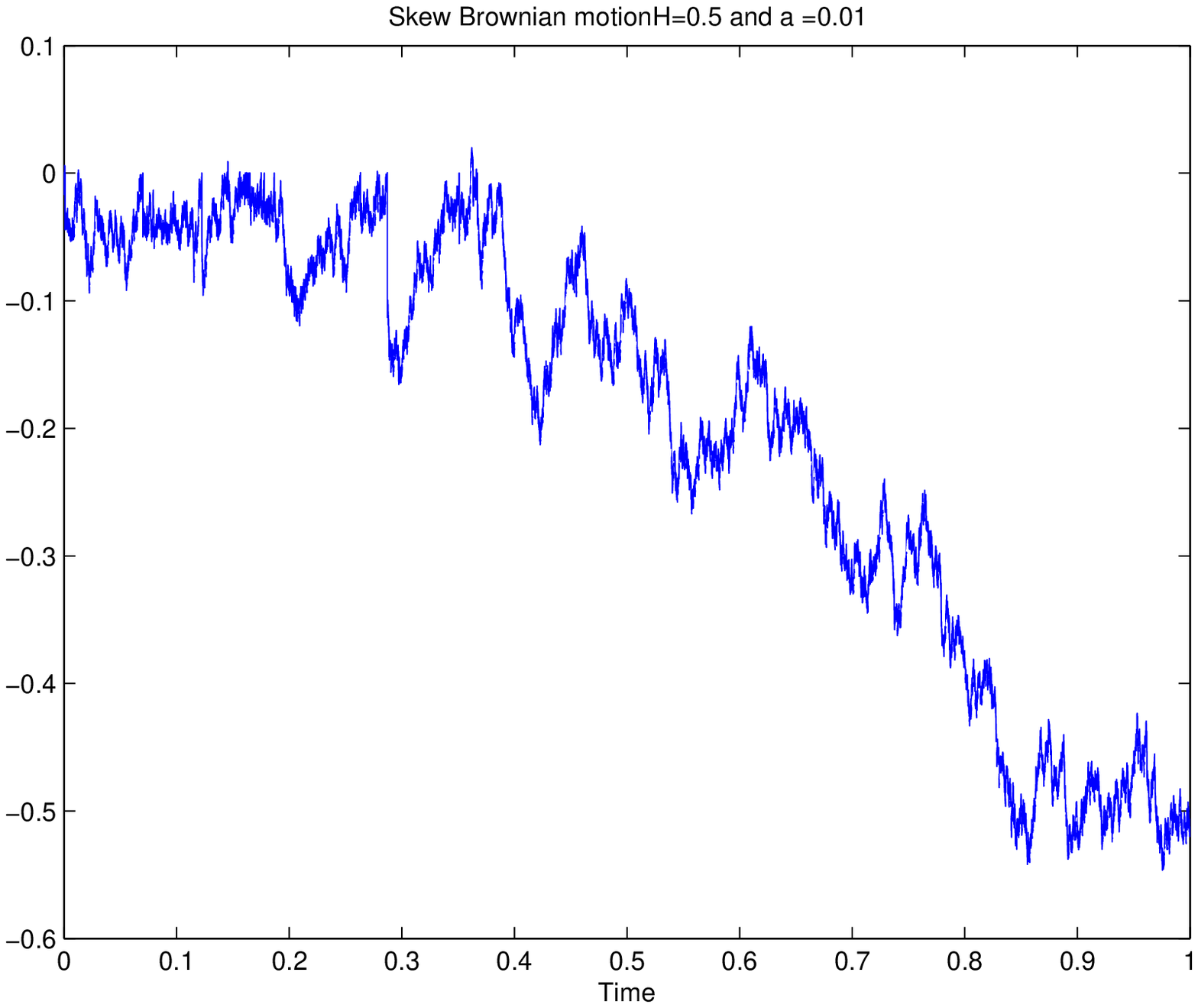}}
\subfigure[SBm for $\alpha = 0.5$.]{\includegraphics[width=45mm]{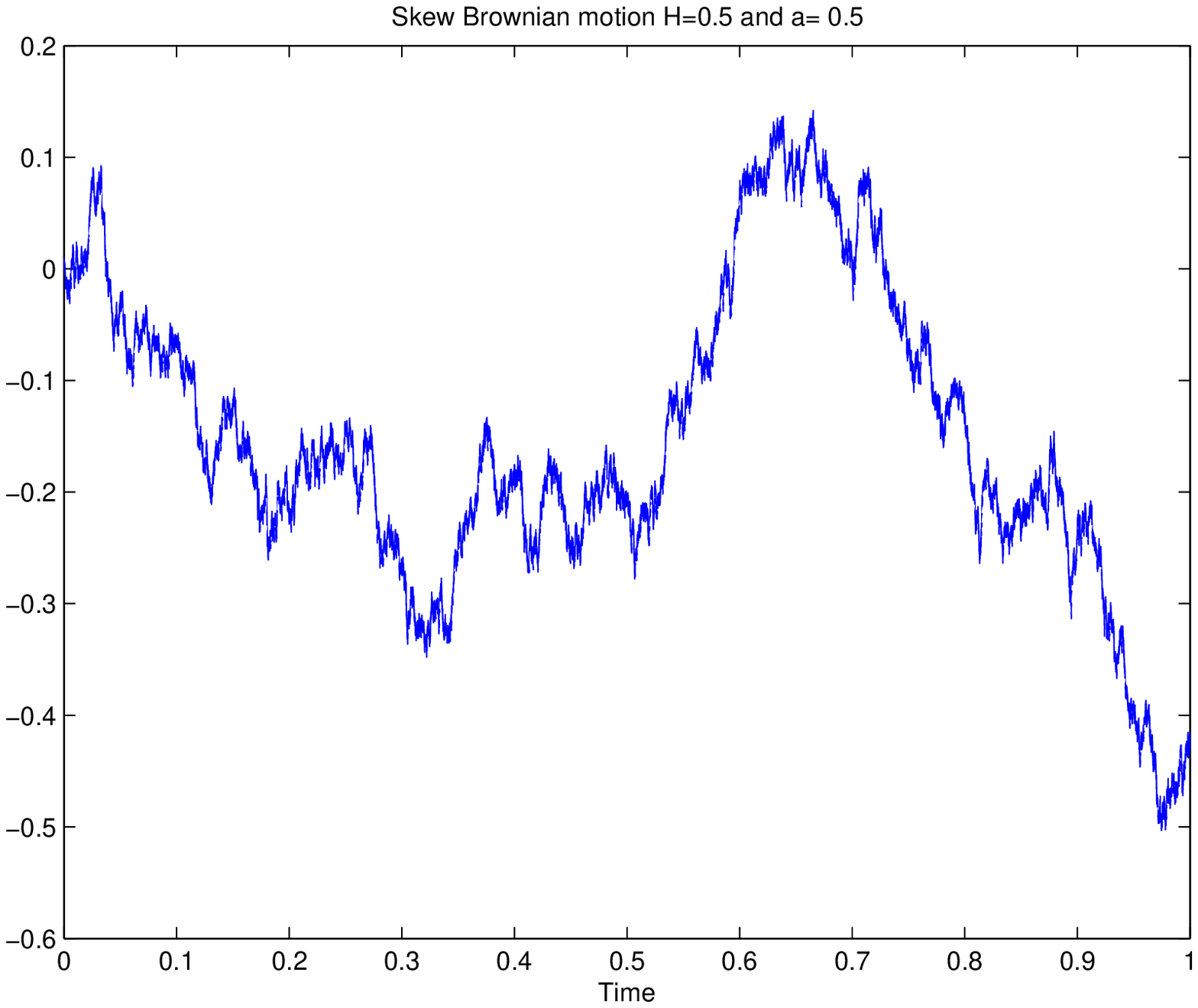}}
\label{SBm}
\end{figure}

The  following pictures  shows  samples for  the unique solution of (\ref{DISC}), for different values of $H$ and $\alpha$. 
\begin{figure}[htbp]
\centering
\subfigure[$x_t$ for  $H=0.75$ and  $\alpha = 0.1$]{\includegraphics[width=45mm] {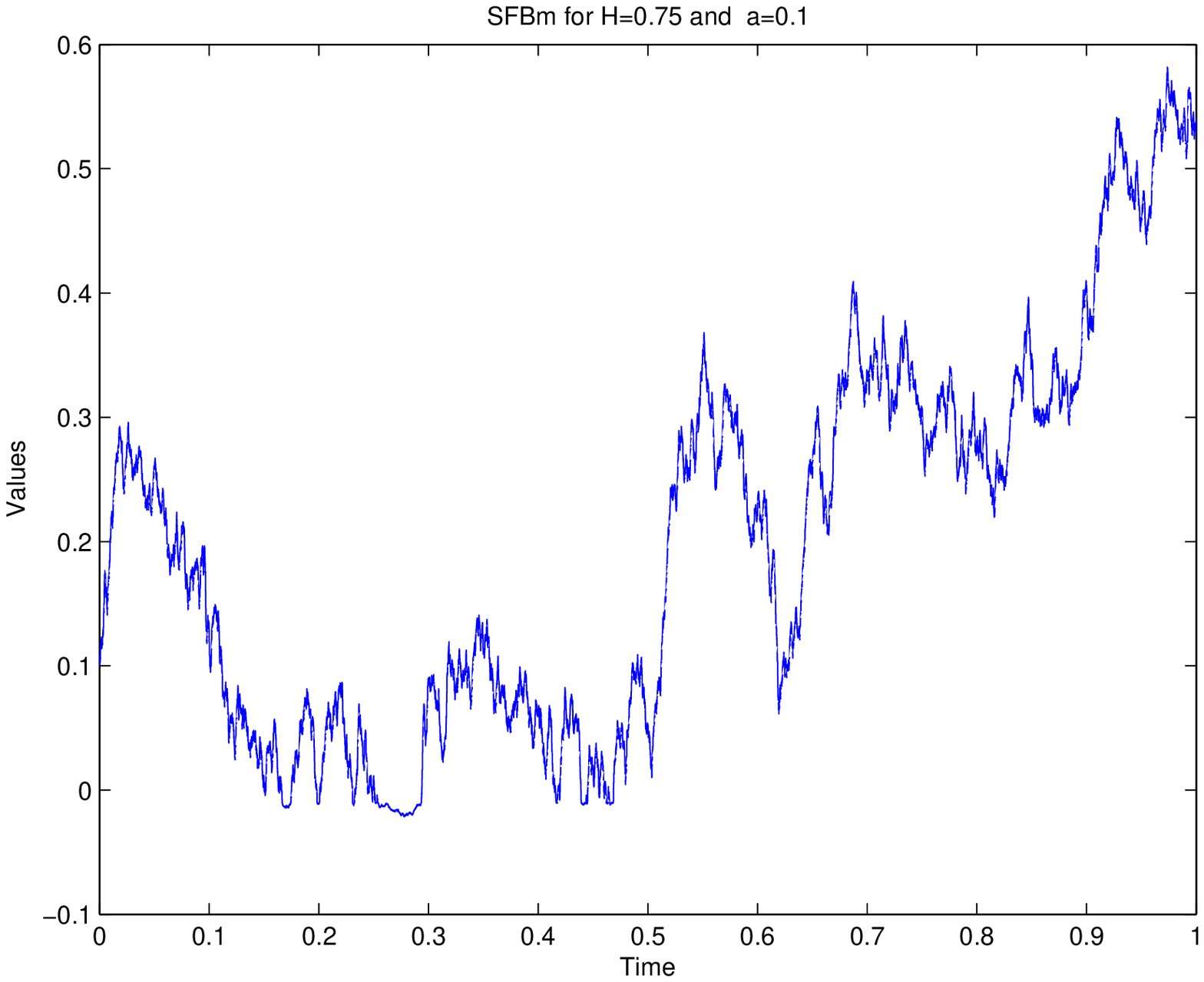}}
\subfigure[$x_t$ for $H=0.75$ and  $\alpha = 0.5$]{\includegraphics[width=45mm]{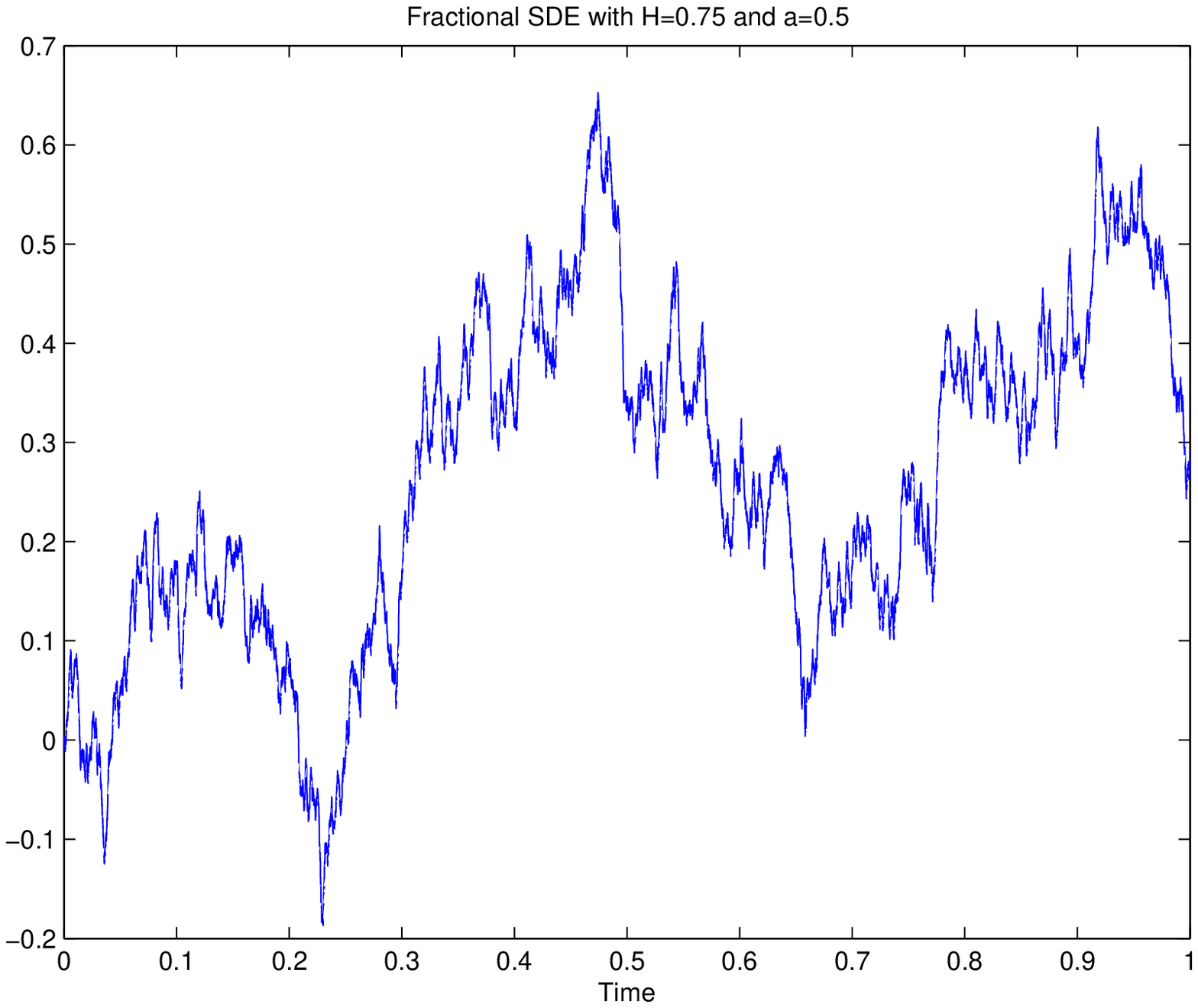}}
\subfigure[$x_t$ for $H=0.75$ and  $\alpha = 0.99$]{\includegraphics[width=45mm]{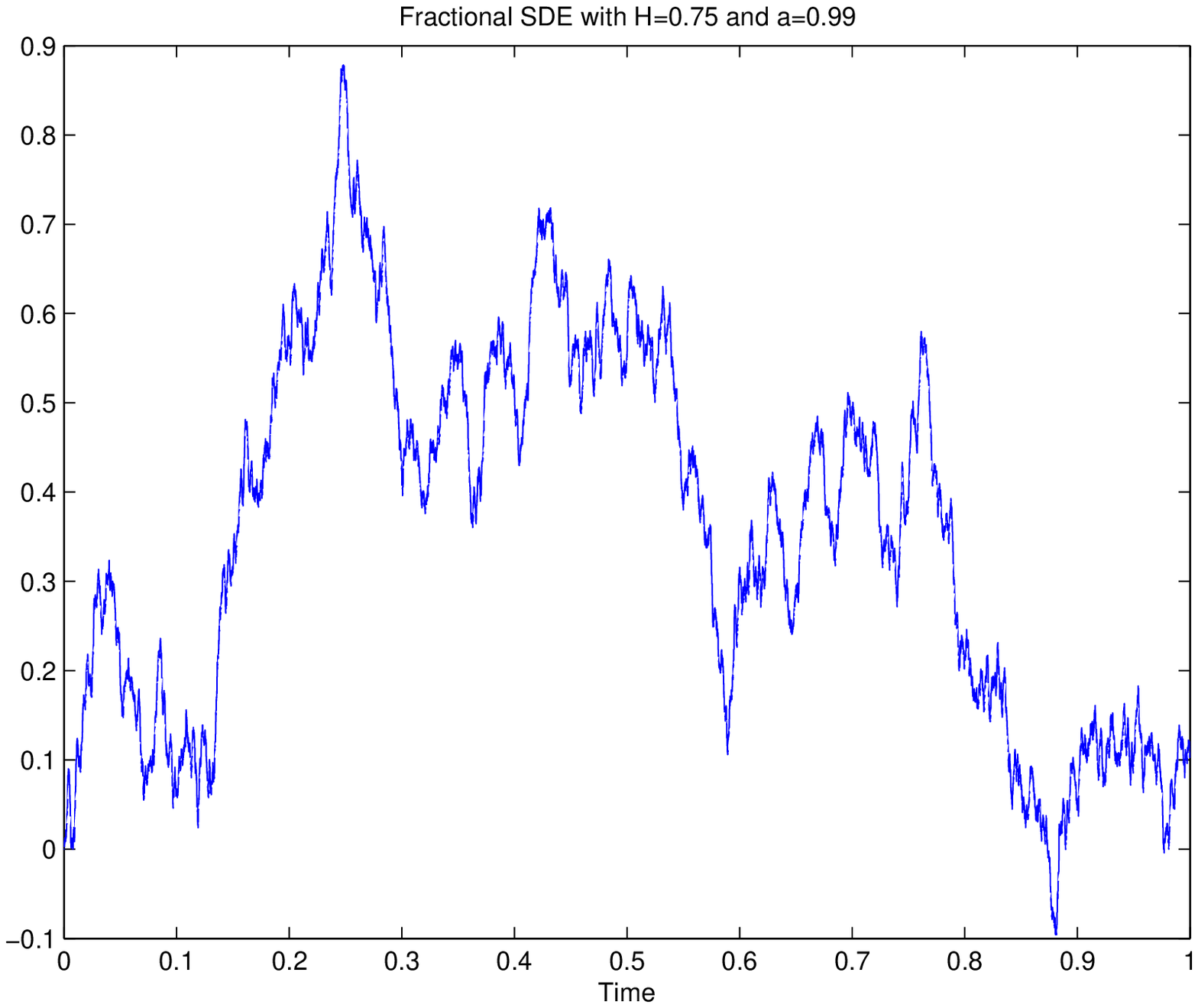}}
\label{}
\end{figure}

\begin{figure}[htbp]
\centering
\subfigure[$x_t$ for $H=0.95$ and  $\alpha = 0.1$]{\includegraphics[width=45mm] {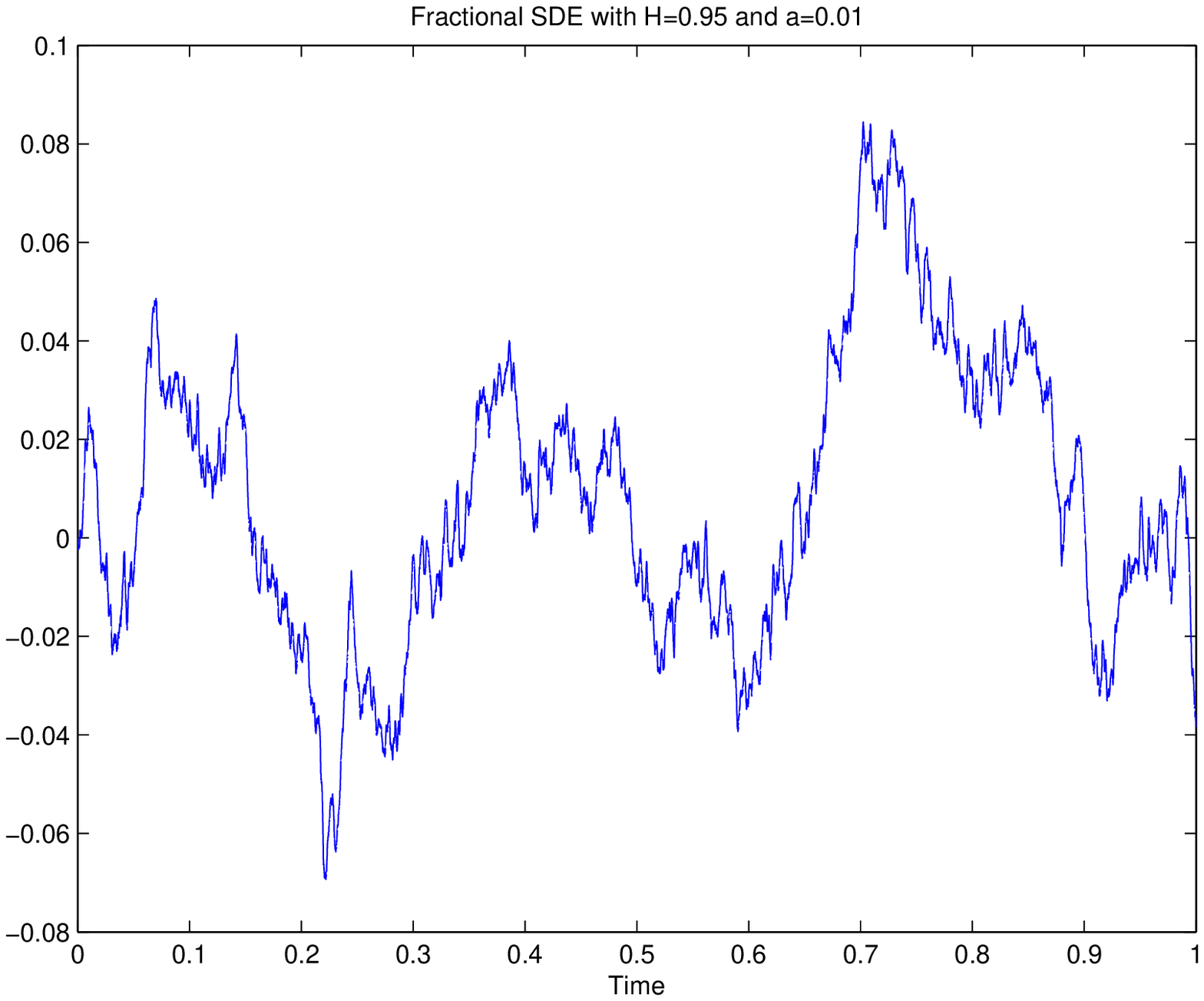}}
\subfigure[$x_t$ for $H=0.95$ and  $\alpha = 0.5$]{\includegraphics[width=45mm]{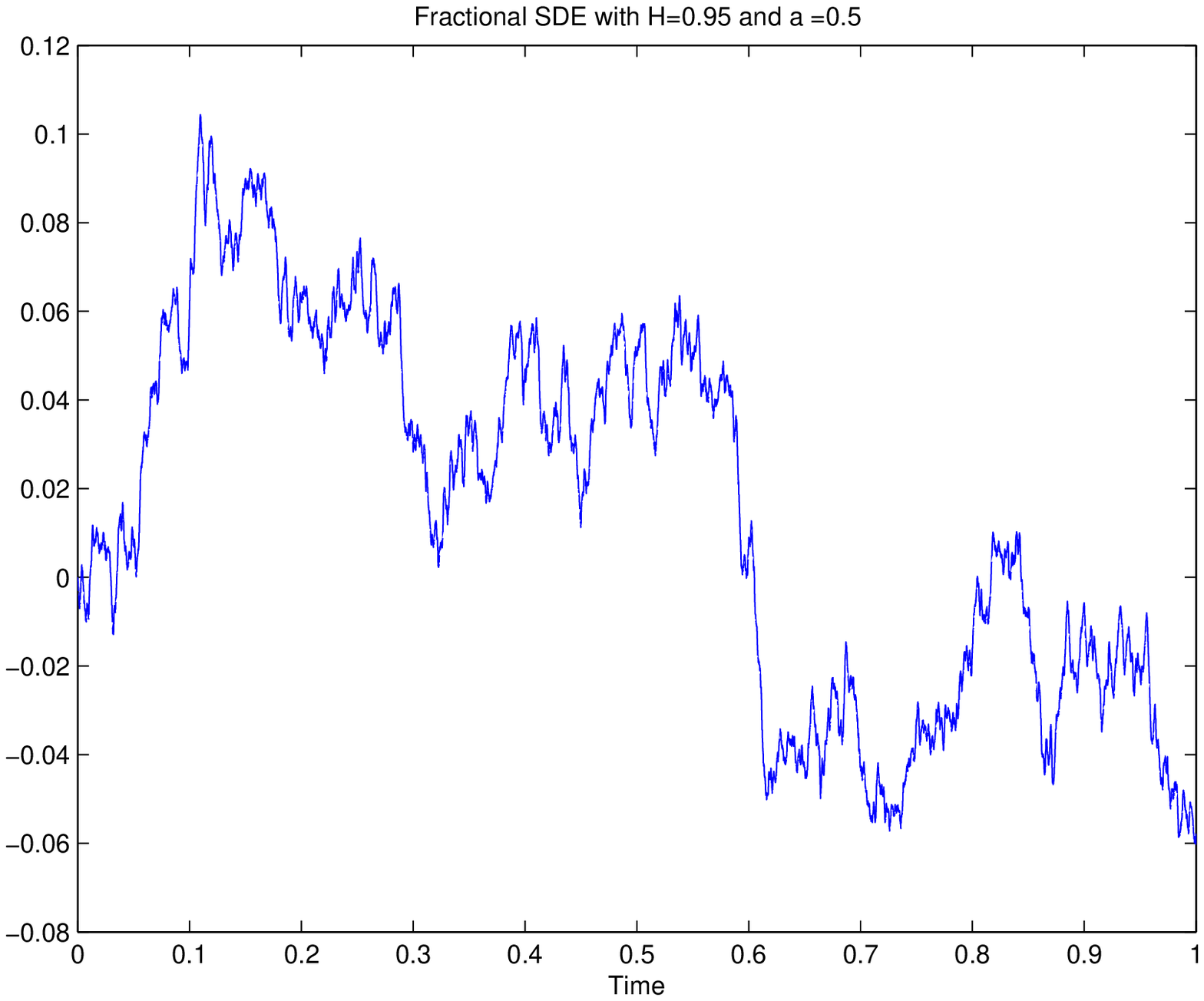}}
\subfigure[$x_t$ for $H=0.95$ and  $\alpha = 0.99$]{\includegraphics[width=45mm]{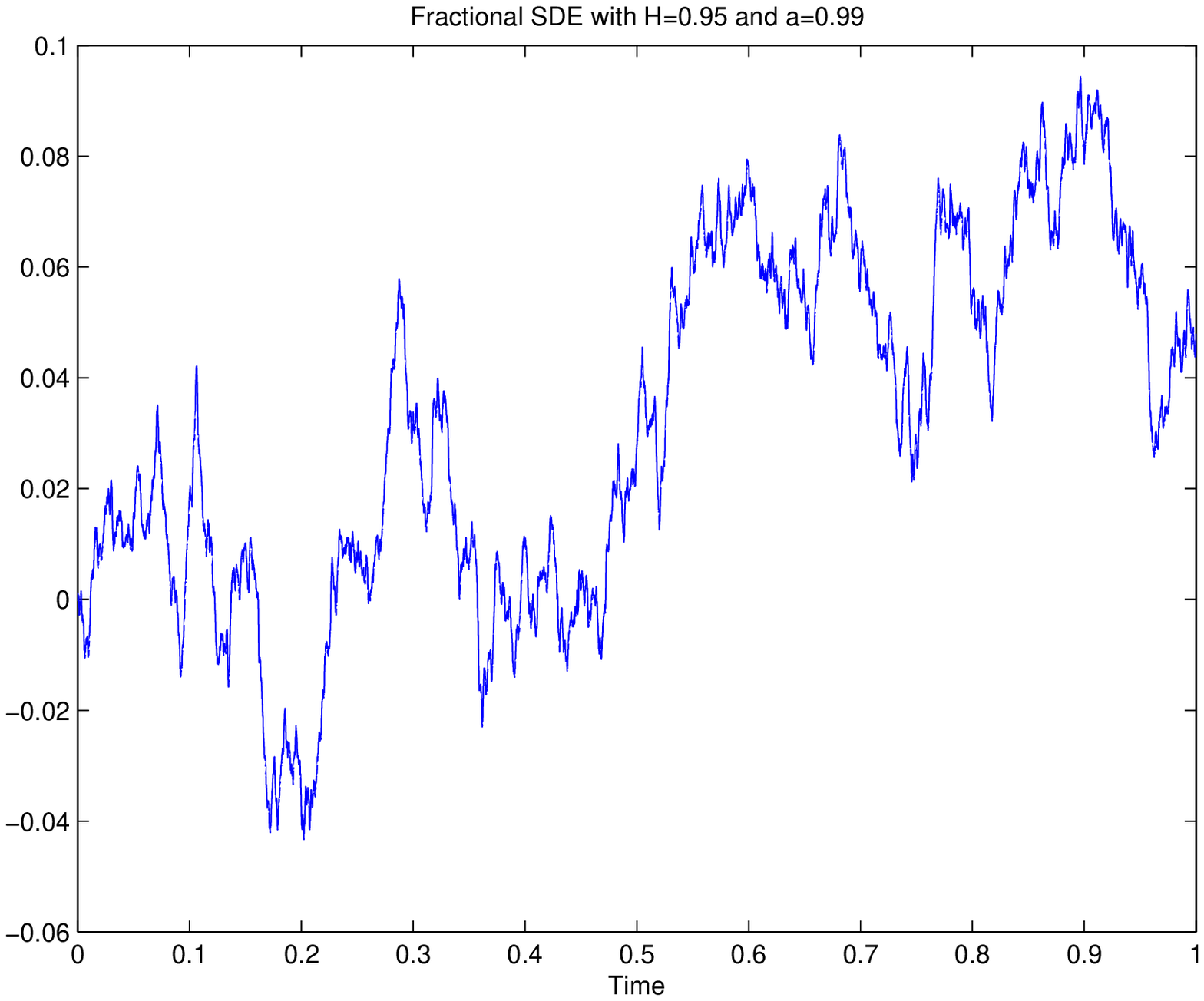}}
\label{}
\end{figure}

\newpage 
{\bf Acknowledgments}

J. Garz\'on was partially supported by the Project HERMES 16930. 
Jorge A. Le\'on was partially supported by 
CONACyT grant 220303. S. Torres was partially supported by the Project PIA ACT1112 C and Fondecyt Grant 1130586. This
work was also developed as part of Inria Chile-CIRICproject  Communication and information research \& innovation center, 10CEII-9157. 
\bibliographystyle{amsplain}

\begin{thebibliography}{99}

\bibitem{AMV} Azmoodeh, Ehsan; Mishura, Yuliya and Valkeila, Esko. On  hedging
European options in  geometric fractional Brownian motion market model. Statistics \&
Decisions  27, (2009), 129-143.



\bibitem{BM}  Bai, Lihua; Ma, Jin.  Stochastic differential equations driven by fractional Brownian motion and Poisson point process. Bernoulli 21 (2015), no. 1, 303-334. 

\bibitem{Lejay} Lejay, Antoine. On the constructions of the skew Brownian motion. Probability Surveys Vol. 3 (2006) 413-466.

\bibitem{MN}  Mishura, Yuliya and    Nualart, David.  Weak solutions for stochastic differential equations with additive fractional noise.  Statistics \& Probability Letters 70 (2004) 253-261

\bibitem{MSV} Mishura, Yuliya; Schevchenko, Georgiy and Valkeila, Esko.
Random variables  as pathwise integrals with  respect to fractional Brownian  motion.
Stochastic Processes and their Applications 123 (2013), 2353-2369.

\bibitem{Nakao} Nakao, S. On the pathwise uniqueness of solutions of one-diemensional stochastic differential equations. Osaka J. Math. 9 (1972), 513-518



\bibitem{Ouknine} Ouknine, Y., 1988. Generalisation da un Lemme de S. Nakao et applications. Stochastics 23, 149-157.

\bibitem{Sam} S.G. Samko, A.A. Kilbas and O.I. Marichev: Fractional Integrals and Derivatives.
Theory and Applications. Gordon and Breach Science Publishers (1993).


\bibitem{zahle} M. Z\"ahle, Integration with respect to fractal functions and stochastic calculus. I. {\it Prob. Theory Relat. Fields} {\bf 111}, 333--374, (1998).

%
\end{thebibliography}

\end{document}